\documentclass[11pt]{amsart}

\usepackage{amsmath}

\usepackage{amssymb,euscript,tikz,units}
\usepackage[colorlinks,citecolor=blue,linkcolor=red]{hyperref}
\usepackage{verbatim}
\usepackage[nameinlink]{cleveref}
\usepackage{enumitem}
\usepackage{hyperref}
\usepackage{url}

\newcommand{\calT}{\mathcal{T}}
\newcommand{\T}{\calT}
\newcommand{\M}{\mathcal{M}}
\newcommand{\sM}{\mathcal{M}}

\newcommand{\R}{\mathbb{R}}

\newcommand{\eps}{\epsilon}

\newcommand{\sbs}{\subset}

\newcommand{\E}{\mathbb{E}_{\rm WP}^g}

\newcommand{\eg}{\textit{e.g.\@ }}

\newcommand{\ie}{\textit{i.e.\@}}

\def\sys{\mathop{\rm sys}}
\def\area{\mathop{\rm Area}}
\def\arcsinh{\mathop{\rm arcsinh}}

\def\Vol{\mathop{\rm Vol}}

\def\Prob{\mathop{\rm Prob}\nolimits_{\rm WP}^g}
\def\Mod{\mathop{\rm Mod}}

\theoremstyle{plain}
\newtheorem{theorem}{Theorem}

\newtheorem{proposition}[theorem]{Proposition}
\newtheorem{lemma}[theorem]{Lemma}

\newtheorem{remark}[theorem]{Remark}

\newcommand{\be}{\begin{equation}}
\newcommand{\ene}{\end{equation}}
\newcommand{\br}{\begin{remark}}
\newcommand{\er}{\end{remark}}
\newcommand{\bl}{\begin{lemma}}
\newcommand{\el}{\end{lemma}}
\newcommand{\bcor}{\begin{cor}}
\newcommand{\ecor}{\end{cor}}
\newcommand{\bpro}{\begin{pro}}
\newcommand{\epro}{\end{pro}}
\newcommand{\ben}{\begin{enumerate}}
\newcommand{\een}{\end{enumerate}}
\newcommand{\bp}{\begin{proof}}
\newcommand{\ep}{\end{proof}}
\newcommand{\bpo}{\begin{pro}}
\newcommand{\epo}{\end{pro}}
\newcommand{\beq}{\begin{equation*}}
\newcommand{\eeq}{\end{equation*}}
\newcommand{\bear}{\begin{eqnarray}}
\newcommand{\eear}{\end{eqnarray}}
\newcommand{\beqar}{\begin{eqnarray*}}
\newcommand{\eeqar}{\end{eqnarray*}}
\newcommand{\bt}{\begin{theorem}}
\newcommand{\et}{\end{theorem}}
\newcommand{\bex}{\begin{excer}}
\newcommand{\eex}{\end{excer}}

\theoremstyle{definition}

\theoremstyle{remark}

\newtheorem*{rem*}{Remark}
\newtheorem*{ques*}{Question}
\newtheorem*{def*}{Definition}
\newtheorem*{con*}{Construction}
\newtheorem*{thm*}{\bf Theorem}
\newtheorem*{definition*}{Definition}
\newtheorem*{assum*}{Assumption $(\star)$}
\newtheorem*{obs*}{Observation}

\title[Determinants of Laplacians]{Averages of determinants of Laplacians over moduli spaces for large genus}

\author{Yuxin He}
\address{Yau Mathematical Sciences Center and Department of Mathematical Sciences, Tsinghua University, Beijing, China}
\email{hyx21@mails.tsinghua.edu.cn}
\author{Yunhui Wu}
\address{Yau Mathematical Sciences Center and Department of Mathematical Sciences, Tsinghua University, Beijing, China}
\email{yunhui\_wu@tsinghua.edu.cn}


\begin{document}

\maketitle

\begin{abstract}
Let $\mathcal{M}_g$ be the moduli space of hyperbolic surfaces of genus $g$ endowed with the Weil-Petersson metric. We view the regularized determinant $\log \det(\Delta_{X})$ of Laplacian as a function on $\mathcal{M}_g$ and show that there exists a universal constant $E>0$ such that as $g\to \infty$,
\begin{enumerate}
\item the expected value of $\left|\frac{\log \det(\Delta_{X})}{4\pi(g-1)}-E \right|$ over $\mathcal{M}_g$ has rate of decay $g^{-\delta}$ for some uniform constant $\delta \in (0,1)$;
\item  the expected value of $\left|\frac{\log \det(\Delta_{X})}{4\pi(g-1)}\right|^\beta$ over $\mathcal{M}_g$ approaches to $E^\beta$ whenever $\beta \in [1,2)$.
\end{enumerate}
\end{abstract}

\section{Introduction}
Let $X$ be a closed hyperbolic surface of genus $g$ $(g>1)$. The spectrum of the Laplacian $\Delta_X$ of $X$ on $L^2(X)$ is a discrete closed subset in $\R^{\geq 0}$ and consists of eigenvalues. We enumerate them, counted with multiplicity, in the following increasing order:
\[0=\lambda_0(X)<\lambda_1(X)\leq \lambda_2(X) \leq \cdots \to +\infty.\]
For $z\in \mathbb{C}$, the spectral zeta function of $X$ is given as
\[\zeta_X(z)=\sum_{i=1}^\infty \frac{1}{\lambda_i^z(X)}.\]
From Weyl's Law, the function $\zeta_X(z)$ is well-defined and holomorphic when $\mathrm{Re}(z)>1$. The \emph{regularized determinant} is usually defined by
\[\log \det (\Delta_X)\overset{\text{def}}{=}-\zeta_X'(0),\]
provided that $\zeta_X(\cdot)$ has an analytic extension to $z=0$. It is known from \eg  Hoker-Phong \cite{HP86} and Sarnak \cite{sarnak1987determinants}  that 
 \begin{equation}\label{eq-det-zetaprime}
     \det (\Delta_X)=Z_0^\prime(1)e^{E\cdot 4\pi(g-1)},
 \end{equation}
where 
\[E=\frac{-1+2\log 2\pi+8\xi^\prime(-1)}{8\pi}\approx 0.0538,\]
and $Z_0(s)$ is the Selberg zeta function $$
Z_0(s)=\prod_\gamma\prod_{k=0}^\infty \left(1-e^{(k+s)\ell_\gamma(X)}\right)
$$
where $\gamma$ runs over all primitive closed geodesics on $X$ and $\ell_\gamma(X)$ is the length of $\gamma$ on $X$. It is known by Wolpert \cite{Wol87} that the magnitudes $\left| \log \det(\Delta_X) \right|$ and $\left| \log Z_0'(1) \right|$ are proper functions on moduli space $\sM_g$ of Riemann surfaces of genus $g$, that is, they will be divergent when the surface $X$ goes to the boundary of $\sM_g$.

In this work, we view $\log \det (\Delta_X)$ and $\log Z_0'(1)$ as random variables with respect to the probability measure $\Prob$ on $\sM_g$ given by the Weil-Petersson metric. Naud is the first one to study their asymptotic behaviors for large genus and shows in \cite{naud2023determinants} that

\begin{thm*}[(Naud)]
For any $\eps>0$,
\be\label{e-Naud-1}
\lim\limits_{g\to \infty}\Prob\left(X\in \sM_g; \ \frac{\left| \log Z_0^\prime(1) \right|}{4\pi(g-1)}<\eps\right)=1.
\ene
\end{thm*}

\noindent Actually it is shown in \cite{naud2023determinants} that the property in \eqref{e-Naud-1} also holds for the other two models of random hyperbolic surfaces: random cover \cite{MN20, MNP20} and Brooks-Makover \cite{BM04}. In the proof, Naud essentially only requires uniform spectral gaps and suitable countings of closed geodesics for large genus. In this paper, we focus on the Weil-Petersson model and use the techniques developed in \cite{nie2023large,wu2022random} to show that 
 \begin{theorem}\label{mainthm2}
     There exists a uniform constant $0<\delta<1$ such that  as $g\to\infty$,   
\[\frac{1}{\Vol_{\mathrm{WP}}(\sM_g)}\int_{\sM_g}\frac{\left| \log Z_0^\prime(1) \right|}{4\pi(g-1)} dX=O\left(\frac{1}{g^{\delta}}\right),\]
where the implied constant is universal and $\Vol_{\mathrm{WP}}(\sM_g)$ is the Weil-Petersson volume of $\sM_g$.
 \end{theorem}
\noindent Actually, using our method, the $\delta$ can be taken as greater than $0.1$. It would be \emph{interesting} to study the optimal choice of $\delta$ in Theorem \ref{mainthm2}. By using Markov's inequality, Theorem \ref{mainthm2} clearly implies the result of Naud above. Our proof is very different from \cite{naud2023determinants}. Moreover, together with \eqref{eq-det-zetaprime} Theorem \ref{mainthm2} also implies that
\[\lim\limits_{g\to \infty}\frac{1}{\Vol_{\mathrm{WP}}(\sM_g)}\int_{\sM_g}\frac{\left| \log \det (\Delta_X) \right|}{4\pi(g-1)} dX=E.\]

Indeed, our second result is as follows.
 \begin{theorem}\label{mainthm1}
For any $\beta\in (0,2)$,
\[\lim\limits_{g\to \infty} \frac{1}{\Vol_{\mathrm{WP}}(\sM_g)}\int_{\sM_g}\left|\frac{\log\det (\Delta_X)}{4\pi(g-1)}\right|^\beta dX=E^\beta.\]
For any $\beta\geq 2$,
\[\int_{\sM_g}\left|\log\det (\Delta_X)\right|^\beta dX=\infty.\]
 \end{theorem}
\noindent For any $\beta\in (0,1)$, this is due to Naud \cite{naud2023determinants}.

\subsection*{Notations.} We say two positive functions $$f_1(g)\prec f_2(g)\quad \emph{or} \quad f_2(g)\succ f_1(g) \quad \emph{or} \quad f_1(g)=O(f_2(g))$$ if there exists a universal constant $C>0$, independent of $g$, such that $f_1(g) \leq C \cdot f_2(g)$; and we say $$f_1(g) \asymp f_2(g)$$ if $f_1(g)\prec f_2(g)$ and $f_2(g)\prec f_1(g)$.

\subsection*{Plan of the paper.} In Section \ref{s-pre} we will provide a review of relevant background materials and show that for $\beta>0$, $\E\left[ 
  \left( \frac{1}{\sys(X)} \right)^\beta\right]\asymp 1$ if and only if $\beta\in (0,2)$ which will be applied in the proof of Theorem \ref{mainthm1}. Then by assuming \eqref{e-Naud-1} and closely following \cite{naud2023determinants}, in Section \ref{s-mt1} we will complete the proof of Theorem \ref{mainthm1}. In the last Section, we will use the techniques developed in \cite{nie2023large,wu2022random} to finish the proof of Theorem \ref{mainthm2}.

\subsection*{Acknowledgements.}
We would like to thank all the participants in our seminar on Teichm\"uller theory for helpful discussions and comments on this project. We also would like to thank Fr\'ed\'eric Naud for his interests on this work.

\section{Preliminary}\label{s-pre}
In this section, we set up notations and basic propositions on the regularized determinant of the Laplacian and the Weil-Petersson model of random surfaces.
\subsection{Determinants of Laplacians}
 Let $X_{g,n}$ with $2g-2+n\geq 1$ be a complete hyperbolic surface of genus $g$ with $n$ geodesic boundaries. By the Gauss-Bonnet formula, its hyperbolic area satisfies $\mathrm{Area}(X_{g,n})=2\pi(2g-2+n)$.  For a closed geodesic $\gamma$ on $X_{g,n}$, we use $\ell(\gamma)$ or $\ell_\gamma(X_{g,n})$ to denote its length. 

Now let $X=X_g$ be a closed hyperbolic surface of genus $g$ $(g>1)$ and $\Delta_X$ be its Laplacian. Recall that as in the introduction, the regularized determinant $\log \det (\Delta_X)=-\zeta_X'(0)$. Indeed, it is known from \eg \cite{HP86, sarnak1987determinants, naud2023determinants} that it also satisfies the following identity
\begin{equation}\label{eq-def-logdetlaplace}
\log \det\Delta_X=4\pi(g-1)E+\gamma_0-\int_{0}^1\frac{S_X(t)}{t}dt-\int_1^\infty\frac{S_X(t)-1}{t}dt,
\end{equation}
where $\gamma_0$ is the Euler constant, $E\approx 0.0538$ is the constant given in \eqref{eq-det-zetaprime}, and $$S_X(t)=\frac{e^{-t/4}}{(4\pi t)^\frac{1}{2}}\sum_{k\geq 1}\sum_{\gamma\in \mathcal{P}} \frac{\ell(\gamma)}{2\sinh(k\ell(\gamma)/2)}e^{-(k\ell(\gamma))^2/4t}.$$ Here $\mathcal{P}$ is the set of oriented primitive closed geodesics on $X$.
One may see \cite{naud2023determinants} and references therein for more details.
Selberg's trace formula gives the following identity:
\begin{equation}\label{sel-trace-formula}
    \sum_{j=0}^\infty e^{-t\lambda_j}=S_X(t)+\frac{4\pi(g-1)e^{-\frac{t}{4}}}{4\pi}\int_{-\infty}^{\infty}re^{-tr^2}\tanh(\pi r)dr.
\end{equation}
See \cite[Theorem 9.5.3]{buser2010geometry} 
 or \cite[Theorem 5.6]{bergeron2016spectrum} for a 
general form of Selberg's trace formula. One may also see \eg \cite{selberg1956harmonic} and \cite{hejhal2006selberg}.

\subsection{Counting geodesics} We will apply the following two countings on closed geodesics. The first one is for general closed geodesics. One may see \eg \cite[Lemma 6.6.4]{buser2010geometry}.
\begin{lemma}\label{counting in buser}Let $S$ be a hyperbolic surface of genus $g\geq 2$ and let $L>0$. There are at most $(g-1)e^{L+6}$ oriented closed geodesics on $S$ of length $\leq L$ which are not iterates of closed geodesics of length $\leq 2\arcsinh 1$.
\end{lemma}

Recall that a  closed geodesic $\gamma$ is called \emph{filling} in a hyperbolic surface $Y$ with geodesic boundaries if each component of $Y\setminus \gamma$ is homotopic to a single point or some boundary component of $Y$. The second counting is the following bound for filling closed geodesics that  is essential in the proof of Theorem \ref{mainthm2}. One may see \cite[Theorem 4]{wu2022random} or \cite[Theorem 18]{WX22pgt} for more details. 
\begin{theorem}[(Wu-Xue)]\label{wuxue counting}
For any  $0<\eps<\frac{1}{2}$ and $m=2g-2+n\geq 1$, there exists
	a constant $c(\eps,m)>0$ only depending on $\eps$ and $m$ such that for all $L>0$ and any compact hyperbolic surface $Y$ of genus $g$ with $n$ boundary simple closed geodesics, the following inequality holds:
	$$\#_{f}(Y,L)\leq c(\eps,m)\cdot  e^{L-\frac{1-\eps}{2}\ell(\partial Y)},$$
	where $\#_{f}(Y,L)$ is the number of filling closed geodesics in $Y$ of length $\leq L$, and 
 $\ell(\partial Y)$ is the total length of the boundary closed geodesics of $Y$.
\end{theorem}

\subsection{The Weil-Petersson model of random surfaces}
For $2g-2+n\geq 1$, let $\T_{g,n}$ be the Teichm\"uller space of Riemann surfaces with $g$ genus and $n$ punctures.
The mapping class group  $\Mod_{g,n}$ acts on $\T_{g,n}$ keeping the Weil-Petersson symplectic form $\omega_{\mathrm{WP}}$ invariant. The quotient space  $\M_{g,n}=\T_{g,n}/\Mod_{g,n}$ is the moduli space of Riemann surfaces.  Set $\T_g=\T_{g,0}$ and  $\M_{g}=\M_{g,0}$. For  $L=(L_1,\cdots,L_n)\in \R^n_{\geq 0}$, let $\T_{g,n}(L)$ be the Teichm\"uller space of bordered hyperbolic surfaces with $n$ geodesic boundaries of length $L_1,\cdots,L_n$, and $\M_{g,n}(L)=\T_{g,n}(L)/\Mod_{g,n}$. In particularly, $\T_{g,n}(0,\cdots,0)=\T_{g,n}$.

The Fenchel-Nielsen coordinates corresponding to a pants decomposition $\{\alpha_i\}_{i=1}^{3g-3+n}$ of $X_{g,n}$ on the Teichm\"uller space $\T_{g,n}(L)$ are of the form  $X\mapsto (\ell_{\alpha_i}(X),\tau_{\alpha_i}(X))_{i=1}^{3g-3+n}$. Here $\tau_{\alpha_i}(X)$ is the  measured twist length along $\alpha_i$. Wolpert \cite{Wolpert82} shows that $\omega_{\mathrm{WP}}$  under the Fenchel-Nielsen coordinates is given by the following theorem: 

\begin{theorem}[(Wolpert)]\label{thm Wolpert}
	The Weil-Petersson symplectic form $\omega_{\mathrm{WP}}$ on $\T_{g,n}(L)$ is given by 
	$$\omega_{\mathrm{WP}} = \sum_{i=1}^{3g-3+n} d\ell_{\alpha_i} \wedge d\tau_{\alpha_i}.$$
\end{theorem}
The Weil-Petersson volume form is given by 
$$d\Vol\nolimits_{\mathrm{WP}}=\tfrac{1}{(3g-3+n)!}\underbrace{\omega_{\mathrm{WP}}\wedge\cdots\wedge\omega_{\mathrm{WP}}}_{3g-3+n\ \text{copies}}.$$
This is a measure on $\T_{g,n}(L)$ and invariant under $\Mod_{g,n}$. Then it is also a measure on $\M_{g,n}(L)$, which is still denoted by $d\Vol_{\mathrm{WP}}$ or $dX$ for short. Denote $V_{g,n}(L)$ to be the total volume of $\M_{g,n}(L)$ under the Weil-Petersson metric. $V_{g,n}(L)$ is a polynomial on $L$  by \cite{Mirz07}. Set $V_{g,n}=V_{g,n}(0)$ for short. Now we consider the closed case $\sM_g$.  The Weil-Petersson metric induces the Weil-Petersson probability measure $\Prob$ on $\M_g$ by:
$$\Prob(\mathcal{A}):=\frac{1}{V_{g}}\int_{\M_{g}} \mathbf{1}_{\mathcal{A}} dX$$
where $\mathcal{A}\sbs\M_g$ is a Borel subset and $\mathbf{1}_{\mathcal{A}}$ is its characteristic function. The  expectation  of  a random variable
$f$ is given by $$
\E\left[ f\right]:=\frac{1}{V_g}\int_{\M_g}fdX.
$$
The following integration formula is given by Mirzakhani in \cite{mirzakhani2013growth} which is very useful in the study of Weil-Petersson random hyperbolic surfaces.
\begin{theorem}[(Mirzakhani)]\label{MIF}
    For any multi-curve $\gamma=\cup_{i=1}^k\gamma_i$, there  exists a constant $C_\gamma\in(0,1]$ such that for
 any $F:\mathbb{R}^k\to \mathbb{R}_+,$   $$\begin{aligned}
     &\int_{\sM_{g}}\sum_{\alpha=(\alpha_1,\cdots,\alpha_k)\in\Mod_{g}\cdot \gamma}F\Big(\ell_{\alpha_1}(X),\cdots,\ell_{\alpha_k}(X)\Big)dX\\
=&C_\gamma\int_{x\in\mathbb{R}_+^k}F(x)\Vol(\sM(S_{g}(\gamma)),\ell_{\gamma_i}=x_i)x\cdot dx. \end{aligned}
 $$
Here  $\gamma_i,i=1,\cdots,k$ are mutually disjoint simple closed geodesics,  $x\cdot dx$ represents $x_1\cdots x_kdx_1\cdots dx_k$ and $\Vol(
\sM(S_{g}(\gamma)),\ell_{\gamma_i}=x_i)
$  is the Weil-Petersson volume of the moduli space of Riemann surfaces homotopic to $S_g\setminus \gamma$ with given boundary lengths.
\end{theorem}

\noindent In this paper we study the asymptotic behaviors of the regularized determinants of Laplacians on Weil-Petersson random surfaces. Indeed,
recently the study of spectral theory of Weil-Petersson random surfaces is quite active. One may see \cite{mirzakhani2013growth, MS20, GMST21, Monk22, wu2022random, LW21, AM23, Hide21, Rud22, RW22, WX22pgt, SW22, HT22, naud2023determinants, Gong23, MS23} and references therein for recent developments on related topics.

\subsection{Bounds on Weil-Petersson volumes}
In this subsection we recall several well-known bounds on Weil-Petersson volumes which will be applied later. It is known by \cite[Lemma 22]{nie2023large} or \cite[Proposition 3.1]{MP19} that
\be\label{b-v-1}
\left(1- c(n) \frac{t^2}{g}\right)\cdot \left(\frac{\sinh(t/2)}{t/2} \right)^n\leq \frac{V_{g,n}(t,\cdots,t)}{V_{g,n}}\leq \left(\frac{\sinh(t/2)}{t/2} \right)^n\leq e^{\frac{nt}{2}},
\ene
where $c(n)>0$ is a constant only depending on $n$. It is also known by \cite[Part (3) of Theorem 3.5]{mirzakhani2013growth} that
\[\frac{V_{g-1,2}}{V_g}=1+O\left(\frac{1}{g}\right).\]
And by \cite[Lemma 3.3]{mirzakhani2013growth} we know that
\[\frac{\sum\limits_{i=1}^{\left[\frac{g}{2}\right]}V_{i,1}\cdot V_{g-i,1}}{V_g}\asymp \frac{1}{g}.\]
These two equations above give that
\be\label{b-v-2}
\frac{V_{g-1,2}+\sum\limits_{i=1}^{\left[\frac{g}{2}\right]}V_{i,1}\cdot V_{g-i,1}}{V_g}\asymp 1.
\ene

\subsection{Expectation of $\E\left[ 
  \left( \frac{1}{\sys(X)} \right)^\beta\right]$ }
Denote by $\sys(X)$ the length of a shortest closed geodesic on $X$ which is always simple.
Similar to the proof of \cite[Corollary 4.3]{mirzakhani2013growth}, we prove the following result which will be used in the proof of Theorem \ref{mainthm1}.
\begin{proposition}\label{sysbetacritical}
 For any $\beta>0$, we have
\be
\E\left[ 
  \left( \frac{1}{\sys(X)} \right)^\beta\right] \ \asymp \ \nonumber
\begin{cases}
1,\ &\ \textit{\emph{if}}\  0<\beta<2; \\
+\infty, \ &\ \textit{\emph{if}}\  2\leq \beta. 
\end{cases}
\ene
\end{proposition}

\begin{proof}
For $\beta>0$, we set $f:\M_g\to \mathbb{R}_{\geq 0}$ as $$
 f(X)=\sum_{\ell_{\gamma}(X)\leq 1}\frac{1}{\ell_\gamma(X)^\beta}.
 $$
We firstly consider the case \underline{$0<\beta<2$}. For the lower bound, it follows from \cite[Theorem 4.2]{mirzakhani2013growth} that
\be
\begin{aligned}
\E\left[ 
  \left( \frac{1}{\sys(X)} \right)^\beta\right]&\geq \frac{\Vol_{\mathrm{WP}}(\sM_g^{<1})}{V_g}\\
&\asymp 1, 
\end{aligned}
\ene
where $\sM_g^{<1}=\{X\in \sM_g; \ \sys(X)<1\}$. For the upper bound, we cut $X$ along a systolic geodesic $\gamma$. Then $X\setminus \gamma$ is homotopic to either $S_{g-1,2}$ or $S_{i,1}\cup S_{g-i,1}$ for $1\leq i\leq \left[\frac{g}{2}\right]$. Thus, it follows from Mirzakhani's integration formula, \ie, Theorem \ref{MIF} that
\beqar
&&\E\left[\left(\frac{1}{\sys(X)}\right)^\beta\right]\leq 1+\frac{1}{V_g}\int_{\M_g}f(X)dX\\
&&\prec 1+\frac{1}{V_g}\int_{0}^1\left(V_{g-1,2}(t,t) +\sum_{i=1}^{\left[\frac{g}{2}\right]}V_{i,1}(t)\cdot V_{g-i,1}(t)  \right)\cdot t^{1-\beta}dt.
\eeqar
This together with \eqref{b-v-1} and \eqref{b-v-2} implies that for $0<\beta<2$,
\be
\begin{aligned}
\E\left[\left(\frac{1}{\sys(X)}\right)^\beta\right]&\prec  1+\frac{V_{g-1,2}+\sum_{i=1}^{\left[\frac{g}{2}\right]}V_{i,1}\cdot V_{g-i,1}}{V_g}\int_{0}^1e^t t^{1-\beta}dt\\
&\asymp 1. 
\end{aligned}
\ene

\noindent Now we consider the case \underline{$\beta\geq 2$}. By the standard Collar Lemma (see \eg \cite{keen1974collars} or \cite{buser2010geometry}), there exist at most $3g-3$ different simple closed geodesics on $X\in\M_g$ of length $\leq 1$. Thus,  it follows from Mirzakhani's integration formula, \ie, Theorem \ref{MIF} that
\beqar
\int_{\M_g}\left(\frac{1}{\sys(X)}\right)^\beta dX &\geq& \frac{1}{3g-3}  \int_{\M_g} f(X)dX\\
&\geq & \frac{1}{3g-3} \int_0^1 t^{1-\beta} V_{g-1,2}(t,t)dt.
\eeqar
This together with the lower bound of \eqref{b-v-1} implies that for $\epsilon_0>0$ small enough,
\be
\begin{aligned}
\int_{\M_g}\left(\frac{1}{\sys(X)}\right)^\beta dX &\geq \frac{V_{g-1,2}}{3g-3} \int_0^{\epsilon_0} \left(\frac{\sinh(t/2)}{t} \right)^2\cdot t^{1-\beta}dt\\
&= +\infty.
\end{aligned}
\ene
The proof is complete.
\end{proof}

\section{Proof of Theorem \ref{mainthm1}}\label{s-mt1}
In this section we prove Theorem \ref{mainthm1}, assuming \eqref{e-Naud-1} of Naud. 

Firstly we need the following lemma which is proved in \cite{naud2023determinants} for $0<\beta<1$. 
\begin{lemma}\label{finite-C-beta}
 For any $0<\beta<2$, there exists a constant $C_\beta>0$ only depending on $\beta$ such that for any $g\geq 2,$
$$\E\left[\left|\frac{\log \det (\Delta_X)}{4\pi(g-1)}\right|^\beta\right]<C_\beta.
 $$  
\end{lemma}

\begin{proof}
We closely follow the proof of \cite[Theorem 5.1]{naud2023determinants}. First recall that the following inequality is proved in the proof of \cite[Theorem 5.1]{naud2023determinants} (see Page 284 of \cite{naud2023determinants}):
\[\frac{|\log \det(\Delta_X)|}{4\pi(g-1)}\prec 1+|\log (\lambda_*(X))|+\frac{1}{\sys(X)}+\left(\frac{\log^+\sys^{-1}(X)}{\sys(X)}\frac{N_X^0(1)}{4\pi(g-1)}\right),\]
where  $\lambda_*(X)=\min\{\lambda_1(X),\frac{1}{4}\}$, $\log^+(x)=\max\{\log (x),0\}$  and  $N_X^0(1)$ is the number of primitive closed geodesics on $X$ of length $\leq 1$. For $\beta >0$, by using the elementary inequality that $(a+b)^\beta \leq 2^\beta\cdot (a^\beta+b^\beta) $ for $a,b\geq 0$, we have
\be
\begin{aligned}
&\frac{|\log \det(\Delta_X)|^\beta}{(4\pi(g-1))^\beta}\prec 1+|\log (\lambda_*(X))|^\beta\\
&+\left(\frac{1}{\sys(X)}\right)^\beta+\left(\frac{\log^+\sys^{-1}(X)}{\sys(X)}\frac{N_X^0(1)}{4\pi(g-1)}\right)^\beta.
\end{aligned}
\ene
It is shown in the proof of \cite[Theorem 5.1]{naud2023determinants} (see Page 286 of \cite{naud2023determinants}) that for all $\beta>0$,
\be
\E[ |\log(\lambda_*(X))|^\beta]\prec 1.
\ene
It is known by Proposition \ref{sysbetacritical} that for $0<\beta <2$,
\be
\E\left[ 
  \left( \frac{1}{\sys(X)} \right)^\beta\right] \asymp 1.
\ene
By the standard Collar Lemma (see \eg \cite{keen1974collars} or \cite{buser2010geometry}) we know that
\[ N_X^0(1)\leq 3g-3.\]
Thus, we have
\be
\begin{aligned}
&\E\left[\left(\frac{\log^+\sys^{-1}(X)}{\sys(X)}\frac{N_X^0(1)}{4\pi(g-1)}\right)^\beta \right]\\
 &\prec \E\left[ \left(\frac{\log^+\sys^{-1}(X)}{\sys(X)}\right)^\beta\right]\\
&\prec \E\left[ 
  \left( \frac{1}{\sys(X)} \right)^{\frac{\beta+2}{2}}\right] \\
&\asymp 1,
\end{aligned}
\ene
where we apply Proposition \ref{sysbetacritical} in the last equation. Then the conclusion follows from all the equations above.
\end{proof}

Recall that
\begin{equation}\label{eq-def-logdetlaplace-recall}
\log \det\Delta_X=4\pi E (g-1)+\gamma_0-\int_{0}^1\frac{S_X(t)}{t}dt-\int_1^\infty\frac{S_X(t)-1}{t}dt,
\end{equation}
where $\gamma_0$ is the Euler constant, and $$S_X(t)=\frac{e^{-\frac{t}{4}}}{(4\pi t)^\frac{1}{2}}\sum_{k\geq 1}\sum_{\gamma\in \mathcal{P}} \frac{\ell(\gamma)}{2\sinh\left(\frac{k\ell(\gamma)}{2}\right)}e^{-\frac{(k\ell(\gamma))^2}{4t}}.$$ Here $\mathcal{P}$ is the set of all oriented primitive closed geodesics on $X$. The following lemma gives a lower bound of $\left|  \log\det(\Delta_X) \right|$ when $X$ goes to the boundary of $\M_g$.  One may see \cite{Wol87} for an asymptotic form.
\begin{lemma}\label{lowbound}
There exist two uniform constants $D_1, D_2>0$ such that 
$$
-\log \det (\Delta_X)\geq D_1\cdot \left( \sum\limits_{\ell(\gamma)\leq 1}\frac{1}{\ell(\gamma)}\right)-4\pi(D_2+E)(g-1)-\gamma_0.
$$    
\end{lemma}
\begin{proof}
Define  $$
 \tilde{G}(u)=\int_0^1t^{-\frac{3}{2}}e^{-\frac{1}{4}t-\frac{u^2}{4t}}dt.$$
Then we have $$
\int_0^1 \frac{S_X(t)}{t}dt=\sum\limits_{m,\gamma}\frac{\ell(\gamma)}{4\sqrt{\pi}\sinh(\frac{m\ell(\gamma)}{2})}\tilde{G}(m\ell(\gamma)).
$$
Since $$e^{\frac{1}{4}}\tilde{G}(u)\geq \int_0^1t^{-\frac{3}{2}}e^{-\frac{u^2}{4t}}dt=\frac{4}{u}\int_{\frac{u}{2}}^\infty e^{-x^2}dx\geq \frac{4}{u}\int_{\frac{u}{2}}^{\frac{u+2}{2}} e^{-x^2}dx
\geq \frac{4}{u}e^{-\frac{(u+2)^2}{4}},
$$
we have  \begin{equation}\label{ineq01}
    \int_0^1 \frac{S_X(t)}{t}dt\geq D_1\cdot \left( \sum\limits_{\ell(\gamma)\leq 1}\frac{1}{\ell(\gamma)}\right)
\end{equation}
for some uniform constant $D_1>0$.
By \eqref{sel-trace-formula}, we have 
\[ \frac{S_X(t)-1}{t}\geq -(g-1)\cdot \frac{e^{-\frac{t}{4}}}{t} \cdot \int_{-\infty}^{\infty}re^{-tr^2}\tanh(\pi r)dr.\]
This implies
\begin{equation}\label{ineq1infty}
    \int_1^\infty \frac{S_X(t)-1}{t}dt\geq - D_2 \cdot (g-1)
\end{equation} for some uniform  constant $D_2>0$.
    Then the conclusion follows from \eqref{eq-def-logdetlaplace-recall},  \eqref{ineq01} and \eqref{ineq1infty}.
\end{proof}

\begin{proof}[Proof of Theorem \ref{mainthm1}]
 For $\beta\geq 2$, by Proposition \ref{sysbetacritical} and  Lemma \ref{lowbound} we have that the integral of $\left|\log \det (\Delta_X) 
 \right|^\beta$
diverges. Now we  assume that $0<\beta<2$. For any $\eps>0$, we define
$$\mathcal{A}_g(\eps):=\left\{X\in \sM_g; \ \frac{\log \det (\Delta_X)}{4\pi(g-1)}\in (E-\eps, E+\eps)\right\}.$$
By the result \eqref{e-Naud-1} of Naud,
\be
\lim\limits_{g\to \infty}\Prob(\mathcal{A}_g(\eps))=1. \nonumber
\ene
Since $\eps>0$ is arbitrary, we clearly have
\be 
\liminf\limits_{g\to\infty}\E\left[\left|\frac{\log\det(\Delta_X)}{4\pi(g-1)}\right|^\beta \right]\geq E^\beta.
\ene

\noindent Next we prove the other side bound. Set $\chi_{\mathcal{A}_g^c(\eps)}$  to be the characteristic function of the complement $\mathcal{A}_g^c(\eps)$ of  $\mathcal{A}_g(\eps)\subset \sM_g$.
By H\"older's inequality and Lemma \ref{finite-C-beta} we have that as $g\to \infty$,

\be\label{e-o1}
\begin{aligned}
&\frac{1}{V_g}\int_{\mathcal{A}_g^c(\epsilon)}\left|\frac{\log\det(\Delta_X)}{4\pi(g-1)}\right|^\beta dX=\E\left[\chi_{\mathcal{A}_g^c(\eps)}\cdot\left|\frac{\log\det(\Delta_X)}{4\pi(g-1)}\right|^\beta \right]\\
& \leq  \E\left[\chi_{\mathcal{A}_g^c(\eps)}^p\right]^{\frac{1}{p}}\cdot \E\left[ \left|\frac{\log\det(\Delta_X)}{4\pi(g-1)}\right|^{q\beta}\right]^{\frac{1}{q}}\\
&\prec \left(\Prob(\mathcal{A}_g^c(\eps))\right)^{\frac{1}{p}}\\
&=o(1),
\end{aligned}
\ene
where $\frac{1}{p}+\frac{1}{q}=1$ and $q\cdot\beta<2$. Then it follows from \eqref{e-o1} that
\be
\begin{aligned}
&\E\left[\left|\frac{\log\det(\Delta_X)}{4\pi(g-1)}\right|^\beta \right]=\frac{1}{V_g}\int_{\mathcal{A}_g(\epsilon)}\left|\frac{\log\det(\Delta_X)}{4\pi(g-1)}\right|^\beta dX\\
&+\frac{1}{V_g}\int_{\mathcal{A}_g^c(\epsilon)}\left|\frac{\log\det(\Delta_X)}{4\pi(g-1)}\right|^\beta dX \\
&\leq (E+\eps)^\beta+o(1).
\end{aligned} \nonumber
\ene
By letting $g\to \infty$ and $\eps\to 0$, we get
\be
\limsup\limits_{g\to\infty}\E\left[\left|\frac{\log\det(\Delta_X)}{4\pi(g-1)}\right|^\beta \right]\leq E^\beta.
\ene
The proof is complete.
\end{proof}

\section{Proof of Theorem \ref{mainthm2}}\label{s-mt2}

In this section we prove Theorem \ref{mainthm2}.

By \eqref{eq-det-zetaprime} and \eqref{eq-def-logdetlaplace}, we have 
\be\label{e-z0-1}
\log Z_0^\prime(1)=\gamma_0-\int_0^1\frac{S_X(t)}{t}dt-\int_1^\infty\frac{S_X(t)-1}{t}dt.
\ene
Let $C>0$ and $L=2\arcsinh 1$. 
For any $ \eta\in (0, \frac{3}{16}) $ and $\alpha \in (0,1)$, following \cite{naud2023determinants}, we define 
\be
A(g):=\left\{X\in \sM_g;\ \begin{aligned}
		&\lambda_1(X)\geq \eta \\
		\#\{(\gamma,m)\in \mathcal{P}\times \mathbb{N},\ & m\ell_\gamma(X)\leq L\}\leq C\cdot g^\alpha
	\end{aligned}\right\}.
\ene
Now we recall two results in \cite{wu2022random} and \cite{naud2023determinants} to bound the size of $A(g)$. First, it is shown by Wu-Xue (see \cite[Subsection 7.3]{wu2022random} or \cite[Theorem 29]{WX22pgt}) that for any $\delta>0$,
\begin{equation}\label{Ag-1}
    \Prob\left(X\in \sM_g; \lambda_1(X)< \eta \right)=O_{\delta,\eta}\left(g^{1+\delta-4\sqrt{\frac{1}{4}-\eta}}\right),
\end{equation}
where the implied constant only depends on $\delta$ and $\eta$. Second, it is shown by Naud in \cite[Section 4.3]  {naud2023determinants} that for any $\epsilon>0$,
\[
\begin{aligned}
&\Prob\left(    \#\{(\gamma,m)\in \mathcal{P}\times \mathbb{N},\  m\ell_\gamma(X)\leq L\} <2L\area(X)^{2\epsilon}   \right)\\
&=1-O\left(\frac{1}{\area(X)^{\epsilon}}\right),
\end{aligned}
\]
where $\area(X)=4\pi(g-1)$. Then it follows that for any $\epsilon<\frac{\alpha}{2}$, if  $g$  is 
large enough,
\begin{equation}\label{Ag-2}
    \Prob\left(    \#\{(\gamma,m)\in \mathcal{P}\times \mathbb{N},\  m\ell_\gamma(X)\leq L\} \geq C\cdot g^\alpha \right)=O\left(\frac{1}{g^{\epsilon}}\right).
\end{equation}

\noindent Combining \eqref{Ag-1} and \eqref{Ag-2},  we have that there exists a positive constant $\epsilon_0=\eps_0(\eta,\alpha)$ depending only on $\eta$ and $\alpha$ such that
\be\label{e-l-1}
\Prob\left(A(g)\right)=1-O\left(\frac{1}{g^{\epsilon_0}}\right).
\ene

\noindent By \eqref{e-z0-1}, we have the following natural upper bound for $\E\left[\left|\frac{\log Z_0^\prime(1)}{4\pi(g-1)}\right|\right]$:
\begin{equation}\label{eq-sep log Z_0^prime}
    \begin{aligned}
       & \E\left[\left|\frac{\log Z_0^\prime(1)}{4\pi(g-1)}\right|\right]=\frac{1}{V_g}\int_{\M_g}\left|\frac{\log Z_0^\prime(1)}{4\pi(g-1)} \right| dX\\
        \leq & \underbrace{\frac{1}{V_g}\int_{A(g)^c}\left|\frac{\log Z_0^\prime(1)}{4\pi(g-1)}\right| dX}_{\rm I}+\frac{\Vol(A(g))}{V_g}\cdot \frac{\gamma_0}{4\pi(g-1)}\\
        +&\underbrace{\frac{1}{V_g}\frac{1}{4\pi(g-1)}\int_{A(g)}\int_{R(g)}^\infty \left|\frac{S_X(t)-1}{t}\right|dtdX}_{\rm{II}}+\frac{\log R(g)}{4\pi(g-1)}\frac{\Vol(A(g))}{V_g}\\
         +&\underbrace{\frac{1}{V_g}\frac{1}{4\pi(g-1)}\int_{A(g)}\int_0^{R(g)} \frac{S_X(t)}{t}dtdX}_{\rm{III}},\\
    \end{aligned}
\end{equation}
for some $R(g)>1$ that is to be determined later.

\subsection{Estimate of (\rm{I})} Recall that by \eqref{eq-det-zetaprime} we have
\[\frac{\log Z_0^\prime(1)}{4\pi(g-1)}=\frac{\log\det(\Delta_X)}{4\pi(g-1)}-E.\]
Then   it follows from Holder's inequality, Lemma \ref{finite-C-beta} and \eqref{e-l-1} that given any fixed $q\in (2,\infty),\ p\in (1,2)$ with $\frac{1}{q}+\frac{1}{p}=1$,  for large enough $g$,
\begin{equation}\label{bound-(a)}
\begin{aligned}
    & \frac{1}{V_g} \int_{A(g)^c}\left|\frac{\log Z_0^\prime(1)}{4\pi(g-1)}\right|dX\\
    \leq & \left(
    \Prob\left(A(g)^c\right)
    \right)^{\frac{1}{q}} \cdot \left(  
 \frac{1}{V_g}\int_{\M_g} \left|\frac{\log\det(\Delta_X)}{4\pi(g-1)}-E\right|^p dX\right)^{\frac{1}{p}}\\
\leq &\left(
    \Prob\left(A(g)^c\right)
    \right)^{\frac{1}{q}} \cdot \left(  
 \frac{1}{V_g}\int_{\M_g} 2^p \cdot \left(\left|\frac{\log\det(\Delta_X)}{4\pi(g-1)}\right|^p+E^p\right) dX\right)^{\frac{1}{p}}\\
    \prec & \frac{1}{g^{\frac{\eps_0}{q}}}.\\
\end{aligned}
\end{equation}

\subsection{Estimate of (\rm{II})} 
First by  \cite[Lemma 3.3]{naud2023determinants} we know that for $X\in A(g)$, $$|S_X(t)-1|\prec 4\pi(g-1)e^{-\eta t}.$$
It follows that for $R(g)>1$, 
\beqar
\int_{R(g)}^\infty \left|\frac{S_X(t)-1}{t}\right|dt&\prec& 4\pi(g-1)\int_{\eta R(g)}^\infty\frac{e^{-s}}{s}ds\\
&\prec& 4\pi(g-1)\frac{e^{-\eta R(g)}}{\eta R(g)}.
\eeqar
Therefore we have  
\begin{equation}\label{bound-(c)}
\frac{1}{V_g}\frac{1}{4\pi(g-1)}\int_{A(g)}\int_{R(g)}^\infty \left|\frac{S_X(t)-1}{t}\right|dtdX\prec  \frac{e^{-\eta R(g)}}{\eta R(g)}.
\end{equation}

\subsection{Estimate of (\rm{III})}
 Recall that
\[S_X(t)=\frac{e^{-\frac{t}{4}}}{(4\pi t)^\frac{1}{2}}\sum_{k\geq 1}\sum_{\gamma\in \mathcal{P}} \frac{\ell(\gamma)}{2\sinh\left(\frac{k\ell(\gamma)}{2}\right)}e^{-\frac{(k\ell(\gamma))^2}{4t}}.\]
Set $$L_g=[4 R(g)+2]$$ where $[\cdot]$ is the largest integer part. Now we separate $S_X(t)$ into three parts:
\bear
&&S_X^{L_g+}(t):=\frac{e^{-\frac{t}{4}}}{(4\pi t)^\frac{1}{2}}\sum_{m\ell(\gamma)>L_g} \frac{\ell(\gamma)}{2\sinh\left(\frac{m\ell(\gamma)}{2}\right)}e^{-\frac{(m\ell(\gamma))^2}{4t}},\\
&&S_X^{s,L_g-}(t):=\frac{e^{-\frac{t}{4}}}{(4\pi t)^\frac{1}{2}}\sum_{\emph{simple} \ \gamma, \ m\ell(\gamma)\leq L_g} \frac{\ell(\gamma)}{2\sinh\left(\frac{m\ell(\gamma)}{2}\right)}e^{-\frac{(m\ell(\gamma))^2}{4t}},\\
&&S_X^{ns,L_g-}(t):=\frac{e^{-\frac{t}{4}}}{(4\pi t)^\frac{1}{2}}\sum_{\emph{non-simple} \ \gamma, \ m\ell(\gamma)\leq L_g} \frac{\ell(\gamma)}{2\sinh\left(\frac{m\ell(\gamma)}{2}\right)}e^{-\frac{(m\ell(\gamma))^2}{4t}}.
\eear
Clearly we have
\be\label{s-s-x}
S_X(t)=S_X^{L_g+}(t)+S_X^{s,L_g-}(t)+S_X^{ns,L_g-}(t).
\ene
The following inequality will be applied several times in this subsection: for any $(m,\gamma)\in \mathbb{N}\times\mathcal{P}$, we have\be\label{deal-single-0-R(g)}
\begin{aligned}
    &\int_0^{R(g)}\frac{e^{-\frac{(m\ell(\gamma))^2}{4t}}}{t^{\frac{3}{2}}}dt\\
=&\frac{4}{m\ell(\gamma)}\int_{\frac{m\ell(\gamma)}{2\sqrt{R(g)}}}^\infty e^{-x^2}dx\\
\prec&\frac{1}{m\ell(\gamma)}e^{-\frac{(m\ell(\gamma))^2}{4R(g)}}.
\end{aligned}
\ene
Here we take the substitution of $x=\frac{m\ell(\gamma)}{2\sqrt{t}}$ in the first equation.

\subsubsection{Long closed geodesics.}
First, by \eqref{deal-single-0-R(g)} we have 
\begin{equation}\label{bound-L_g+(1)}
 \begin{aligned}
     &\frac{1}{4\pi(g-1)}\int_0^{R(g)}\frac{S_X^{L_g+}(t)}{t}dt\\
     = &\frac{1}{4\pi(g-1)}\sum_{m\ell(\gamma)\geq L_g}\int_0^{R(g)}\frac{e^{-\frac{t}{4}}}{t(4\pi t)^{\frac{1}{2}}}\frac{\ell(\gamma)}{2\sinh\frac{m\ell(\gamma)}{2}}e^{-\frac{(m\ell(\gamma))^2}{4t}}dt\\
     \prec&    \frac{1}{4\pi(g-1)}\sum_{m\ell(\gamma)\geq L_g} \frac{\ell(\gamma)}{2\sinh\frac{m\ell(\gamma)}{2}}\int_0^{R(g)}\frac{e^{-\frac{(m\ell(\gamma))^2}{4t}}}{t^{\frac{3}{2}}}dt
     \\
     \prec&\frac{1}{4\pi(g-1)}\sum_{m\ell(\gamma)\geq L_g}\frac{1}{2\sinh \frac{m\ell(\gamma)}{2}}e^{-\frac{(m\ell(\gamma))^2}{4R(g)}}.\\
\end{aligned}\end{equation}
Since $X\in A(g)$, 
 it follows from \cite[Lemma 3.2]{naud2023determinants} that for all $k\geq L_g$,
 $$\#\{(m,\gamma)\in \mathbb{N}\times\mathcal{P}:k\leq m\ell(\gamma)\leq k+1\}\prec 4\pi(g-1)e^k.$$ Then we have \begin{equation}\label{bound-L_g+(2)}
     \begin{aligned}
         &\frac{1}{4\pi(g-1)}\sum_{m\ell(\gamma)\geq L_g}\frac{1}{2\sinh \frac{m\ell(\gamma)}{2}}e^{-\frac{(m\ell(\gamma))^2}{4R(g)}}\\
         \leq &
         \frac{1}{4\pi(g-1)}\sum_{k=L_g}^\infty
          \frac{1}{2\sinh \frac{k}{2}}\cdot e^{-\frac{k^2}{4R(g)}}\cdot \#\{(m,\gamma)\in \mathbb{N}\times\mathcal{P}:k\leq m\ell(\gamma)\leq k+1\}\\
          \prec &
         \sum_{k=L_g}^\infty
          e^{-\frac{k^2}{4R(g)}}\cdot e^{\frac{1}{2}k}\\
          \prec&\int_{L_g-1}^\infty e^{-\frac{t^2}{4R(g)}+\frac{t}{2}}dt \\
\prec & \int_{L_g-R(g)-1}^\infty e^{-\frac{s^2}{4R(g)}+\frac{R(g)}{4}} ds\\
          \prec& \sqrt{R(g)} e^{\frac{R(g)}{4}-\frac{(L_g-R(g)-1)^2}{4R(g)}}\\
     \end{aligned}
 \end{equation} 
 Since $L_g>4R(g)+1$,
by \eqref{bound-L_g+(1)} and \eqref{bound-L_g+(2)} we have \begin{equation}\label{bound-L_g+}
    \left| \frac{1}{V_g}\frac{1}{4\pi(g-1)}\int_{A(g)}\int_0^{R(g)} \frac{S_X^{L_g+}(t)}{t}dtdX  \right|\prec \sqrt{R(g)} e^{-2R(g)}.
\end{equation}

\subsubsection{Simple closed geodesics.}
Recall that for a simple closed  geodesic $\gamma$,  $X\backslash\gamma$ is homotopic to either $S_{g-1,2}$ or $S_{i,1}\cup S_{g-i,1}$ for $1\leq i\leq \left[\frac{g}{2}\right]$.  Then it follows from Mirzakhani's integration formula, \ie, Theorem \ref{MIF} that for any $t>0$, 
\be \nonumber
\begin{aligned}
&\frac{1}{V_g}\int_{A(g)}S_X^{s,L_g-}(t)dX\\
\leq & \frac{1}{V_g}\int_{\M_g}S_X^{s,L_g-}(t)dX\\
\leq&\sum_{k}\frac{e^{-\frac{t}{4}}}{(4\pi t)^{\frac12}}\int_0^\infty \frac{x}{2\sinh \frac{kx}{2}}e^{-\frac{(kx)^2}{4t}}\\
&\times \frac{V_{g-1,2}(x,x)+\sum_{i=1}^{\left[\frac{g}{2}\right]}V_{i,1}(x)V_{g-i,1}(x)}{V_g} xdx.
\end{aligned}
\ene
For any $k>0$ and $x>0$, it is clear that $\frac{\sinh \left(\frac{x}{2}\right)}{\sinh \left(\frac{kx}{2}\right)}\leq \frac{1}{k}$. Thus, the above inequality together with \eqref{b-v-1} and \eqref{b-v-2} implies that 
    $$
\begin{aligned}
   &\frac{1}{V_g}\int_{A(g)}S_X^{s,L_g-}(t)dX\\
   \prec&\sum_{k}\frac{e^{-\frac{t}{4}}}{t^{\frac12}}\int_0^\infty \frac{1}{\sinh \frac{kx}{2}}e^{-\frac{(kx)^2}{4t}}\sinh^2\left(\frac{x}{2}\right)dx\\
    \prec&\sum_{k}\frac{e^{-\frac{t}{4}}}{ t^{\frac12}k}\int_0^\infty \left( 
   e^{-\left(\frac{kx}{2\sqrt{t}}-\frac{\sqrt{t}}{2k}\right)^2+
   \frac{t}{4k^2}}-e^{-\left(\frac{kx}{2\sqrt{t}}+\frac{\sqrt{t}}{2k}\right)^2+
   \frac{t}{4k^2}}
   \right)dx\\
   =&\sum_{k}\frac{e^{-\frac{t}{4}+\frac{t}{4k^2}}}{ t^{\frac12}k}\frac{2\sqrt{t}}{k}\int_{-\frac{\sqrt{t}}{2k}}^{\frac{\sqrt{t}}{2k}}e^{-s^2}ds\\
   \prec &\sum_{k}\frac{e^{-\frac{t}{4}+\frac{t}{4k^2}}\sqrt{t}}{k^3}\\
   \prec &\sqrt{t}.
\end{aligned}
$$
Therefore, it follows that \begin{equation}  \label{bound simple L_g-} 
\begin{aligned}
&\frac{1}{4\pi(g-1)}\frac{1}{V_g}\int_{A(g)}  \int_0^{R(g)}\frac{S_X^{s,L_g-}(t)}{t}dt dX\\
\prec & \frac{1}{4\pi(g-1)}\int_0^{R(g)}\frac{\sqrt{t}}{t}dt\\
\asymp & \frac{\sqrt{R(g)}}{g}.   
\end{aligned}
\end{equation}

\subsubsection{Non-simple closed geodesics.} For any non-simple  closed primitive geodesic $\gamma$ on $X$, $\ell(\gamma)\geq 4\arcsinh 1$ (see \eg \cite{buser2010geometry}). Then we have  
\begin{equation}\label{bound nonsimple SXt}
  \begin{aligned}
     S_X^{ns,L_g-}(t)  \prec & \frac{e^{-\frac{t}{4}}}{t^\frac{1}{2}}\sum_{m\ell(\gamma)\leq L_g}\frac{\ell(\gamma)e^{-\frac{(m\ell(\gamma))^2}{4t}}}{e^{\frac{m\ell(\gamma)}{2}}}\\
     \prec&  \frac{e^{-\frac{t}{4}}}{t^\frac{1}{2}}\sum_{\ell(\gamma)\leq L_g}\sum_{m\geq 1}\frac{\ell(\gamma)e^{-\frac{\ell^2(\gamma)}{4t}}}{e^{\frac{m\ell(\gamma)}{2}}}\\
\asymp&\frac{e^{-\frac{t}{4}}}{t^\frac{1}{2}}\sum_{\ell(\gamma)\leq L_g}\ell(\gamma)e^{-\frac{\ell^2(\gamma)}{4t}-\frac{\ell(\gamma)}{2}},
  \end{aligned}\nonumber
   \end{equation}
  where all the $\gamma's$ are taken over all non-simple primitive closed geodesics on $X$.  
By \eqref{deal-single-0-R(g)} we have
\begin{equation}\label{bound int sxt}
      \begin{aligned}
 \int_0^{R(g)}\frac{S_X^{ns,L_g-}}{t}dt \prec& \sum_{\ell(\gamma)\leq L_g}\ell(\gamma)e^{-\frac{\ell(\gamma)}{2}}\int_0^{R(g)}\frac{e^{-\frac{\ell^2(\gamma)}{4t}}}{t^{\frac{3}{2}}}dt\\
\prec& \sum_{\ell(\gamma)\leq L_g}e^{-\frac{\ell(\gamma)}{2}-\frac{\ell^2(\gamma)}{4R(g)}}\\
=&\sum_{\ell(\gamma)\leq L_g}e^{ -\frac{(\ell(\gamma)-R(g))^2}{4R(g)}+\frac{R(g)}{4}-\ell(\gamma)}\\
\leq&\sum_{\ell(\gamma)\leq L_g}e^{\frac{R(g)}{4}-\ell(\gamma)}.
\end{aligned}
  \end{equation}
As in \cite{MP19, nie2023large,wu2022random}, for each non-simple closed geodesic $\gamma\subset X$, one may let $Y\subset X$ be a subsurface of geodesic boundary such that $\gamma$ is filling in $Y$. Then for each non-simple closed geodesic $\gamma$ with $\ell(\gamma)\leq L_g$, by using elementary isoperimetric inequality, it is shown in \eg \cite[Proposition 7]{wu2022random} that
\[\ell(\partial Y)\leq 2L_g \quad \textit{and} \quad \area(Y)\leq 4L_g.\]
Now we count $\gamma$ according to the topology type of $Y$. Fix $\epsilon_1\in(0,1)$,  for $Y\simeq S_{g_0,k}$ with $1\leq |\chi(S_{g_0,k})|\leq 16$, by Theorem \ref{wuxue counting} we know that the number of filling geodesics in $Y$ with length $\leq T$ satisfies $$
 \#_f(Y,T)\leq c(\epsilon_1)e^{T-\frac{1-\epsilon_1}{2}\ell(\partial 
 Y)}
 $$
 where $c(\epsilon_1)>0$ only depends on $\epsilon_1$.
 Then we have
\begin{equation}\label{bound 1-16}
 \begin{aligned}
  \sum_{\substack{\ell(\gamma)\leq L_g\\
\gamma \textit{ fills }Y\simeq S_{g_0,k}}}e^{\frac{R(g)}{4}-\ell(\gamma)}=&  \sum_{Y\simeq S_{g_0,k}}\sum_{\gamma \textit{ fills } Y}e^{\frac{R(g)}{4}-\ell(\gamma)}1_{[0,L_g]}(\ell(\gamma))\\
   \leq &\sum_{Y\simeq S_{g_0,k}} \ \sum_{m=0}^{[L_g]} \ \sum_{\substack{m\leq \ell(\gamma)<m+1\\
\gamma \textit{ fills }Y}} e^{\frac{R(g)}{4}-\ell(\gamma)}1_{[0,L_g]}(\ell(\gamma))\\
   \prec &\sum_{Y\simeq S_{g_0,k}} \ \sum_{m=0}^{[L_g]}  e^{\frac{R(g)}{4}-m}\cdot e^{m+1-\frac{1-\epsilon_1}{2}\ell(\partial Y)}1_{[0,2L_g]}(\ell(\partial 
 Y))\\
   \prec & L_g\cdot \sum_{Y\simeq S_{g_0,k}}   e^{\frac{R(g)}{4}-\frac{1-\epsilon_1}{2}\ell(\partial Y)}1_{[0,2L_g]}(\ell(\partial 
 Y)).\\
   \end{aligned}
 \end{equation}
It is known from \cite[Proposition 35]{wu2022random} that  for any $1\leq |2g_0-2+k|\leq 16$ and $0<\eps_1<1$, we have
\begin{equation}\label{bound 1-16-step2}
    \begin{aligned}
        \frac{1}{V_g}\int_{\M_g}\sum_{Y\simeq S_{g_0,k}}   e^{-\frac{1-\epsilon_1}{2}\ell(\partial Y)}1_{[0,2L_g]}(\ell(\partial 
 Y))dX\prec \frac{L_g^{66}e^{\eps_1L_g} }{g}.
    \end{aligned}
\end{equation}
Combining \eqref{bound 1-16} and \eqref{bound 1-16-step2}, we have 
\begin{equation}\label{int 1-16}
 \begin{aligned}
 \frac{1}{V_g}\int_{\M_g}\sum_{1\leq |\chi(Y)|\leq 16} \sum_{\gamma \textit{ fills } Y} e^{\frac{R(g)}{4}-\ell(\gamma)}  1_{[0,L_g]}(\ell(\gamma)) dX \prec \frac{L_g^{67}e^{\frac{R(g)}{4}+\epsilon_1L_g}}{g}.
   \end{aligned}
 \end{equation}
For $17\leq |\chi(S_{g_0,k})|\leq [\frac{4L_g}{2\pi}]$, it follows from Lemma \ref{counting in buser} that the number of filling closed geodesics on $Y$ with length $\leq T$ satisfies $$
\#_f(Y,T)\leq (g-1)e^{T+6}.
$$
Then \begin{equation}\label{bound 17+}
\begin{aligned}
   &  \sum_{Y\simeq S_{g_0,k}}\sum_{\gamma \textit{ fills } Y}e^{\frac{R(g)}{4}-\ell(\gamma)}1_{[0,L_g]}(\ell(\gamma))\\
   \prec& \sum_{Y\simeq S_{g_0,k}} \sum_{m=0}^{[L_g]}\sum_{\substack{m\leq \ell(\gamma)<m+1\\
\gamma \textit{ fills } Y}} e^{\frac{R(g)}{4}-\ell(\gamma)} 1_{[0,2L_g]}(\ell(\partial 
 Y))\\
   \prec &g \sum_{Y\simeq S_{g_0,k}} \sum_{m=0}^{[L_g]} e^{\frac{R(g)}{4}-m}\cdot e^m 1_{[0,2L_g]}(\ell(\partial 
 Y))\\
   \prec & g\sum_{Y\simeq S_{g_0,k}}  L_g  e^{\frac{R(g)}{4}}1_{[0,2L_g]}(\ell(\partial 
 Y))\\
 \prec & ge^{\frac{L_g}{2}} L_g\sum_{Y\simeq S_{g_0,k}}   e^{\frac{R(g)}{4}-\frac{\ell(\partial Y)}{4}}1_{[0,2L_g]}(\ell(\partial 
 Y)).
   \end{aligned}
\end{equation}
By \cite[Proposition 33]{wu2022random} we have \begin{equation}\label{bound 17+ step2}
   \frac{1}{V_g} \int_{\M_g} \sum_{17\leq |\chi(Y)|\leq \left[\frac{4L_g}{2\pi}\right]}e^{-\frac{\ell(\partial Y)}{4}}1_{[0,2L_g]}(\ell(\partial 
 Y)) dX\prec\frac{L_ge^{\frac{7}{2}L_g}}{g^{17}}.
\end{equation}
Combining \eqref{bound 17+} and \eqref{bound 17+ step2}, we have \begin{equation}\label{int 17+}
\begin{aligned}
 &\frac{1}{V_g} \int_{\M_g} \sum_{17\leq |\chi(Y)|\leq \left[\frac{4L_g}{2\pi}\right]} \sum_{\gamma \textit{ fills } Y}e^{\frac{R(g)}{4}-\ell(\gamma)}1_{[0,L_g]}(\ell(\gamma)) dX\\
 \prec &\frac{L_g^2e^{\frac{R(g)}{4}+4L_g}}{g^{16}}.
 \end{aligned}
\end{equation}
Finally, combining \eqref{bound int sxt}, \eqref{int 1-16}, and \eqref{int 17+}, we have that for any $0<\epsilon_1<1$,
 \begin{equation}\label{bound nonsimple L_g-}
\begin{aligned} & \frac{1}{V_g} \frac{1}{4\pi(g-1)}\int_{A(g)} \int_0^{R(g)}\frac{S_X^{ns,L_g-}(t)}{t}dt dX \\
   \prec & \frac{L_g^{67}e^{\frac{R(g)}{4}+ \epsilon_1L_g } }{g^2}+\frac{L_g^2e^{\frac{R(g)}{4}+4L_g}}{g^{17}}.
\end{aligned}
\end{equation}

\subsubsection{Upper bound of \rm{(III)}} Since $S_X(t)=S_X^{L_g+}(t)+S_X^{s,L_g-}(t)+S_X^{ns,L_g-}(t)$,
 if we take $L_g=[4R(g)+2]$,  by \eqref{bound-L_g+}, \eqref{bound simple L_g-}  and \eqref{bound nonsimple L_g-} we have that for any $\epsilon_1>0$, \begin{equation}\label{bound-(e)}
     \begin{aligned}
 &\left|\frac{1}{V_g}\int_{A(g)} \frac{1}{4\pi(g-1)} \int_0^{R(g)}\frac{S_X(t)}{t}dt dX\right|\\
 \prec &\sqrt{R(g)}e^{-2R(g)}+\frac{\sqrt{R(g)}}{g}+\frac{R(g)^{67}e^{\left(\frac{1}{4}+4\epsilon_1\right)R(g)}}{g^2}+\frac{R(g)^2e^{\frac{65}{4}R(g)}}{g^{17}}.
  \end{aligned}
\end{equation}
\begin{rem*}
    Actually if we take $L_g=[kR(g)+2]$ for $k>2$, the argument above also works.
\end{rem*}

\subsection{The Endgame}  Now we finish the proof of Theorem \ref{mainthm2}.\begin{proof}[Proof of Theorem \ref{mainthm2}]
We fix $\eta\in(0,\frac{3}{16})$, $\alpha\in(0,1)$ and $R(g)=a\log g$ for some $a>0$. Now we bound each term in \eqref{eq-sep log Z_0^prime}. 
By
\eqref{bound-(a)} we have \be\label{final-I}
({\rm{I}})\prec g^{-\frac{\epsilon_0}{q}},\ene
where $\epsilon_0=\epsilon(\eta,\alpha)>0$ and $q>2$. By \eqref{bound-(c)} we have \be\label{final_II}
({\rm{II}})\prec g^{-a \eta }.\ene 
By \eqref{bound-(e)}, for any $\epsilon_1>0$ we have
\be\label{final-III}
({\rm{III}})\prec g^{-2a+\epsilon_1}+g^{-1+\epsilon_1}+g^{-2+\frac{1}{4}a+\epsilon_1}+g^{-17+\frac{65}{4}a+\epsilon_1}.
\ene

\noindent Now \eqref{eq-sep log Z_0^prime} together with \eqref{final-I}, \eqref{final_II} and \eqref{final-III} implies that if we take $0<a<\frac{68}{65}$, then for any $$0<\delta<\min\left\{1,\frac{\epsilon_0}{q},a\eta,2-\frac{1}{4}a,17-\frac{65}{4}a\right\},$$ we obtain $$
\E\left[\frac{\left| \log Z_0^\prime(1) \right|}{4\pi(g-1)}\right]\prec g^{-\delta}
$$
as desired.
\end{proof}

\bibliographystyle{amsalpha}
\bibliography{Reference}

\providecommand{\bysame}{\leavevmode\hbox to3em{\hrulefill}\thinspace}
\providecommand{\MR}{\relax\ifhmode\unskip\space\fi MR }
\providecommand{\MRhref}[2]{%
  \href{http://www.ams.org/mathscinet-getitem?mr=#1}{#2}
}
\providecommand{\href}[2]{#2}
\begin{thebibliography}{GLMST21}

\bibitem[AM23]{AM23}
Nalini {Anantharaman} and Laura {Monk}, \emph{{Friedman-Ramanujan functions in
  random hyperbolic geometry and application to spectral gaps}}, arXiv e-prints
  (2023), arXiv:2304.02678.

\bibitem[Ber16]{bergeron2016spectrum}
Nicolas Bergeron, \emph{The spectrum of hyperbolic surfaces}, Springer, 2016.

\bibitem[BM04]{BM04}
Robert Brooks and Eran Makover, \emph{Random construction of {R}iemann
  surfaces}, J. Differential Geom. \textbf{68} (2004), no.~1, 121--157.

\bibitem[Bus10]{buser2010geometry}
Peter Buser, \emph{Geometry and spectra of compact riemann surfaces}, Springer
  Science \& Business Media, 2010.

\bibitem[DP86]{HP86}
Eric D'Hoker and D.~H. Phong, \emph{On determinants of {L}aplacians on
  {R}iemann surfaces}, Comm. Math. Phys. \textbf{104} (1986), no.~4, 537--545.

\bibitem[GLMST21]{GMST21}
Clifford Gilmore, Etienne Le~Masson, Tuomas Sahlsten, and Joe Thomas,
  \emph{Short geodesic loops and {$L^p$} norms of eigenfunctions on large genus
  random surfaces}, Geom. Funct. Anal. \textbf{31} (2021), no.~1, 62--110.

\bibitem[Gon24]{Gong23}
Yulin Gong, \emph{Spectral distribution of twisted laplacian on typical
  hyperbolic surfaces of high genus}, Communications in Mathematical Physics
  \textbf{405} (2024), no.~7, 158.

\bibitem[Hej06]{hejhal2006selberg}
Dennis~A Hejhal, \emph{The selberg trace formula for $\rm{PSL} (2,
  \mathbb{R})$: Volume 2}, vol. 1001, Springer, 2006.

\bibitem[Hid23]{Hide21}
Will Hide, \emph{Spectral gap for weil--petersson random surfaces with cusps},
  Int. Math. Res. Not. IMRN \textbf{2023} (2023), no.~20, 17411--17460.

\bibitem[HT25]{HT22}
Will Hide and Joe Thomas, \emph{Short geodesics and small eigenvalues on random
  hyperbolic punctured spheres}, Comment. Math. Helv. \textbf{100} (2025),
  no.~3, 463--506.

\bibitem[Kee74]{keen1974collars}
Linda Keen, \emph{Collars on riemann surfaces}, Discontinuous groups and
  Riemann surfaces (Proc. Conf., Univ. Maryland, College Park, Md., 1973),
  1974, pp.~263--268.

\bibitem[LMS24]{MS20}
Etienne Le~Masson and Tuomas Sahlsten, \emph{Quantum ergodicity for eisenstein
  series on hyperbolic surfaces of large genus}, Math. Ann. \textbf{389}
  (2024), no.~1, 845--898.

\bibitem[LW24]{LW21}
Michael {Lipnowski} and Alex {Wright}, \emph{Towards optimal spectral gaps in
  large genus}, Ann. Probab. \textbf{52} (2024), no.~2, 545--575.

\bibitem[Mir07]{Mirz07}
Maryam Mirzakhani, \emph{Simple geodesics and {W}eil-{P}etersson volumes of
  moduli spaces of bordered {R}iemann surfaces}, Invent. Math. \textbf{167}
  (2007), no.~1, 179--222.

\bibitem[Mir13]{mirzakhani2013growth}
\bysame, \emph{Growth of weil-petersson volumes and random hyperbolic surface
  of large genus}, Journal of Differential Geometry \textbf{94} (2013), no.~2,
  267--300.

\bibitem[MN20]{MN20}
Michael Magee and Fr\'{e}d\'{e}ric Naud, \emph{Explicit spectral gaps for
  random covers of {R}iemann surfaces}, Publ. Math. Inst. Hautes \'{E}tudes
  Sci. \textbf{132} (2020), 137--179.

\bibitem[MNP22]{MNP20}
Michael Magee, Fr\'{e}d\'{e}ric Naud, and Doron Puder, \emph{A random cover of
  a compact hyperbolic surface has relative spectral gap
  {$\frac{3}{16}-\varepsilon$}}, Geom. Funct. Anal. \textbf{32} (2022), no.~3,
  595--661.

\bibitem[Mon22]{Monk22}
Laura Monk, \emph{Benjamini-{S}chramm convergence and spectra of random
  hyperbolic surfaces of high genus}, Anal. PDE \textbf{15} (2022), no.~3,
  727--752.

\bibitem[MP19]{MP19}
Maryam Mirzakhani and Bram Petri, \emph{Lengths of closed geodesics on random
  surfaces of large genus}, Comment. Math. Helv. \textbf{94} (2019), no.~4,
  869--889.

\bibitem[MS25]{MS23}
Laura Monk and Rares Stan, \emph{Spectral convergence of the dirac operator on
  typical hyperbolic surfaces of high genus}, Annales Henri Poincar{\'e},
  vol.~26, Springer, 2025, pp.~365--387.

\bibitem[Nau23]{naud2023determinants}
Fr{\'e}d{\'e}ric Naud, \emph{Determinants of laplacians on random hyperbolic
  surfaces}, Journal d'Analyse Math{\'e}matique \textbf{151} (2023), 265--291.

\bibitem[NWX23]{nie2023large}
Xin Nie, Yunhui Wu, and Yuhao Xue, \emph{Large genus asymptotics for lengths of
  separating closed geodesics on random surfaces}, Journal of Topology
  \textbf{16} (2023), no.~1, 106--175.

\bibitem[{Rud}23]{Rud22}
Ze{\'e}v {Rudnick}, \emph{{GOE statistics on the moduli space of surfaces of
  large genus}}, Geom. Funct. Anal. \textbf{33} (2023), no.~6, 1581--1607.

\bibitem[RW23]{RW22}
Ze\'{e}v Rudnick and Igor Wigman, \emph{On the central limit theorem for linear
  eigenvalue statistics on random surfaces of large genus}, J. Anal. Math.
  \textbf{151} (2023), no.~1, 293--302.

\bibitem[Sar87]{sarnak1987determinants}
Peter Sarnak, \emph{Determinants of laplacians}, Comm. Math. Phys. \textbf{110}
  (1987), no.~1, 113--120.

\bibitem[Sel56]{selberg1956harmonic}
Atle Selberg, \emph{Harmonic analysis and discontinuous groups in weakly
  symmetric spaces with applications to dirichlet series}, J. Indian Math. Soc.
  \textbf{20} (1956), 47--87.

\bibitem[SW25]{SW22}
Yang Shen and Yunhui Wu, \emph{Arbitrarily {S}mall {S}pectral {G}aps for
  {R}andom {H}yperbolic {S}urfaces with {M}any {C}usps}, Comm. Math. Phys.
  \textbf{406} (2025), no.~8, Paper No. 190.

\bibitem[Wol82]{Wolpert82}
Scott Wolpert, \emph{The {F}enchel-{N}ielsen deformation}, Ann. of Math. (2)
  \textbf{115} (1982), no.~3, 501--528.

\bibitem[Wol87]{Wol87}
Scott~A. Wolpert, \emph{Asymptotics of the spectrum and the {S}elberg zeta
  function on the space of {R}iemann surfaces}, Comm. Math. Phys. \textbf{112}
  (1987), no.~2, 283--315.

\bibitem[WX22]{wu2022random}
Yunhui Wu and Yuhao Xue, \emph{Random hyperbolic surfaces of large genus have
  first eigenvalues greater than $\frac{3}{16}-\epsilon$}, Geometric and
  Functional Analysis \textbf{32} (2022), no.~2, 340--410.

\bibitem[WX25]{WX22pgt}
Yunhui {Wu} and Yuhao {Xue}, \emph{{Prime geodesic theorem and closed geodesics
  for large genus}}, J. Eur. Math. Soc. (JEMS) (2025), doi:
  {10.4171/JEMS/1653}.

\end{thebibliography}

\end{document}